\documentclass[a4paper, 11pt]{article}

\usepackage{amssymb, amsbsy, amsmath, amsthm, float}
\usepackage{graphicx}
\usepackage{amsfonts}
\usepackage{latexsym}

\usepackage{psfrag}
\usepackage{setspace}
\usepackage{xcolor}

\usepackage[utf8]{inputenc}
\usepackage{hyperref}

\newtheorem{theorem}{Theorem}

\newtheorem{prop}[theorem]{Proposition}

\newtheorem{lemma}[theorem]{Lemma}

\newcommand{\R}{\mathbb{R}}%
\begin{document}

\title{Tumor containment for Norton-Simon models}

\author{Frank Ernesto Alvarez
\thanks{GMM, INSA Toulouse, 135 Avenue de Rangueil, 31400 Toulouse, France \emph{E-mail:}  alvarez-borg $\alpha\tau$ insa-toulouse.fr}
\and Yannick Viossat\thanks{CEREMADE, Universit\'e Paris Dauphine-PSL, Paris, Place du Mar\'echal de Lattre de Tassigny, F-75775 Paris, France \emph{E-mail:} viossat $\alpha\tau$ ceremade.dauphine.fr
} 
}

\maketitle
\begin{abstract} 
\begin{spacing}{1.2}
\noindent 
Some clinical and pre-clinical data suggests that treating some tumors at a mild, patient-specific dose might delay resistance to treatment and increase survival time. A recent mathematical model with sensitive and resistant tumor cells identified conditions under which a treatment aiming at tumor containment rather than eradication is indeed optimal. This model however neglected mutations from sensitive to resistant cells, and assumed that the growth-rate of sensitive cells is non-increasing in the size of the resistant population. The latter is not true in standard models of chemotherapy. 
This article shows how to dispense with this assumption and allow for mutations from sensitive to resistant cells. This is achieved by a novel mathematical analysis comparing tumor sizes across treatments not as a function of time, but as a function of the resistant population size.

\medskip


\end{spacing}
\end{abstract}

\maketitle

\section{Introduction}

The dominant paradigm in cancer therapy is to treat tumors aggressively. This makes sense if the tumor is curable, but might be counter-productive otherwise. Indeed, tumors contain a large number of cells, some of which may be resistant to treatment. By killing preferentially the most sensitive cells, an aggressive treatment could free resistant cells from competition with sensitive cells, allowing them to develop quickly: a phenomenon called competitive release in ecology (Gatenby 2009b \cite{Gatenby2009b}, Enriquez-Navas et al., 2016 \cite{Enriquez-Navas2016}, Cunningham et al. 2019 \cite{Cunningham2019}).

This led researchers to suggest that, at least for some tumors, treating at, or close to, the maximal tolerated dose should be replaced by treating at the minimal effective dose; that is, the minimal dose that allows to stabilize tumor size, subject to a sufficient quality of life of the patient. The aim is to slow down the growth of resistant cells, by maintaining competition with sensitive cells. 

This idea, which is part of the broader framework of cancer adaptive therapy (Gatenby 2009 \cite{Gatenby2009a}), 
has been tested in vitro, in mice models and on human patients suffering from metastatic castrate-resistant prostate 
cancer  (Gatenby et al. 2009  \cite{Gatenby2009a}, Silva et al. 2012 \cite{Silva2012}, Enriquez-Navas et al. 2016 \cite{Enriquez-Navas2016}, Zhang et al. 2017 \cite{Zhang2017}: trial NCT02415621, Bacevic and Noble et al. 2017 \cite{Bacevic2017a}, Smalley et al. 2019 \cite{Smalley2019}, 
Strobl et al. 2020 \cite{Strobl2020a}, Bondarenko et al. 2021 \cite{Bondarenko2021}, 
Wang et al. 2021a, 2021b \cite{Wang2021a, Wang2021b}, Farrokhian et al. 2022 \cite{Farrokhian2022}). 
Other clinical trials are ongoing or starting in prostate cancer (NCT03511196, NCT05393791), melanoma (NCT03543969), 
rhabdomyosarcoma (NCT04388839) and ovarian cancer (ACTOv/NCT05080556)\footnote{The initial prostate cancer trial has been debated (Mistry 2021 \cite{Mistry2021}, Zhang et al. 2021 \cite{Zhang2021})}.

On the theoretical side, several mathematical models of tumor containment have been studied (e.g., Martin et al. 1992 \cite{Martin1992a}, 
Monro and Gaffney 2009 \cite{Monro2009}, Gatenby et al. 2009 \cite{Gatenby2009a}, Silva et al. 2012 \cite{Silva2012}, Carrère 2017 \cite{Carrere2017}, Zhang et al. 2017 \cite{Zhang2017}, Bacevic and Noble et al. 2017 \cite{Bacevic2017a}, Hansen et al. 2017 \cite{Hansen2017}, Gallaher et al. 2018 \cite{Gallaher2018b}, Cunningham et al. 2018 \cite{Cunningham2018}, Pouchol 2018 \cite{Pouchol2018}, Carrère and Zidani 2020 \cite{Carrere2020}, Strobl et al. 2020 \cite{Strobl2020a}, Cunningham et al. 2020 \cite{Cunningham2020}) leading to the first workshop on Cancer Adaptive Therapy Models (CATMo; \href{https://catmo2020.org/}{https://catmo2020.org/}). However, many of these models make very specific assumptions, e.g., logistic tumor growth with a specific effect of intra-tumor competition and a specific treatment kill-rate (Zhang et al., 2017 \cite{Zhang2017}, Cunningham et al. 2018 \cite{Cunningham2018}, Carrère 2017 \cite{Carrere2017}, Strobl et al. 2020 \cite{Strobl2020a}). This makes it difficult to generalize their conclusions. 

Viossat and Noble (2021) \cite{Viossat2021} recently analysed a more general model with two types of tumor cells: sensitive and fully resistant to treatment. The model takes the form: 
\begin{equation}
\begin{split}
\frac{dS}{dt}(t)  & = S(t) g_S(S(t), R(t), L(t)) \\
\frac{dR}{dt}(t)  & = R(t) g_R (S(t), R(t))
\label{VN_model}
\end{split}
\tag{Model 1}
\end{equation}
where $S(t)$ and $R(t)$ are the total number of sensitive and resistant cells at time $t$, $L(t)$ is the current dose or treatment level, 
and $g_S$ and $g_r$ are per-cell growth-rate functions. They identified qualitative assumptions under which, among other results, 
containing the tumor at its initial size maximizes the time at which the tumor becomes larger than at the beginning of treatment 
(for an idealized form of containment) or is close to maximizing it (for a more realistic form). Similarly, an idealized form of containment 
at a larger threshold size maximizes the time at which tumor size becomes larger than this threshold. By contrast, eliminating all sensitive cells 
at treatment initiation - an idealized form of an aggressive treatment - leads to the quickest time to progression beyond any threshold size, among all treatments that eliminate sensitive cells before this threshold size is crossed.

Some of the assumptions of Viossat and Noble are however debatable. In particular, they assume that the higher the number of resistant cells, the lower the growth rate of sensitive cells. 
Formally, function $g_S$ is non-increasing in $R$.
This assumption helps to compare the size of sensitive
populations across treatments. To see why, assume that
sensitive cells hamper the growth of resistant cells (that is, 
$g_R$ is non-increasing in $S$), and consider two constant dose
treatments, with doses $L_1$ and $L_2 > L_1$, respectively, and
the same initial conditions. Since treatment 2 is more
aggressive, it initially leads to a smaller sensitive
population, hence a larger resistant population than treatment
1: for $t >0$ small enough, $S_2(t) <  S_1(t)$ and 
$R_2(t) \geq R_1(t)$. 
If the growth-rate of sensitive cells $g_S$
is non-increasing in $R$, the fact that treatment 2 is more
aggressive and leads to a larger resistant population both
negatively affect the sensitive population under treatment 2,
ensuring that the sensitive population remains smaller under
treatment 2 than under treatment 1: $S_2 < S_1$. This itself
ensures that $R_2$ remains larger than $R_1$. The inequalities
$S_2 <  S_1$ and $R_2 \geq R_1$ thus propagate, and hold for all
times $t >0$. By contrast, if the growth-rate of sensitive
cells $g_S$ increases with $R$, the fact that $R_2 \geq R_1$ might
boost the growth of sensitive cells under treatment 2, 
even though treatment 2 is more aggressive. But if the
sensitive population becomes larger under treatment 2,  
the inequality $R_2 \geq R_1$ might also cease to hold, 
and the whole argument of Viossat and Noble seems to break. 

Unfortunately, assuming $g_S$ non-increasing in $R$, 
which may seem a natural consequence of competition between tumor cells, is actually problematic. Indeed, it is not satisfied in the Gompertzian model from Monro and Gaffney (2009) \cite{Monro2009} that Viossat and Noble use for simulations: 
\begin{equation}
\begin{split}
\frac{dS}{dt}(t)  & = \rho  \ln(K/N(t)) \, (1 - L(t)) S(t),\\
\frac{dR}{dt}(t)  & = \rho \ln(K/N(t)) \, R(t),
\label{MG_model}
\end{split}
\tag{Model 2}
\end{equation}
where $N(t) = S(t) + R(t)$ is the total number of tumor cells. More precisely, 
in the absence of treatment ($L(t) =0$), the growth-rate of sensitive cells is decreasing in $N$, hence in $R$; 
however, if the treatment level is high enough ($L(t)>1$), the opposite happens, and a large resistant population slows down the regression of the sensitive population.  

This reflects the fact that chemotherapy typically attacks cells that are actively dividing. For various reasons (e.g., boundary growth), a larger tumor size is thought to be associated with a lower growth-fraction, 
i.e., a lower proportion of cells actively dividing (Laird 1964 \cite{Laird1964}; Norton and Simon 1977 \cite{Norton1977}; Gerlee 2013 \cite{Gerlee2013}). Thus the presence of additional resistant cells, 
by making the tumor larger, makes more sensitive cells quiescent, and shields them against the effect of treatment. As a result, the growth rate of sensitive cells is not always decreasing in $R$, 
and the assumptions of Viossat and Noble are not satisfied. The problem occurs for all Norton-Simon models (Norton and Simon 1977 \cite{Norton1977}), where the growth of the sensitive population takes the form:  
$$\frac{dS}{dt}(t) = S(t) g(N(t)) (1 - L(t))$$ 
for some per-cell growth rate function $g$. It also occurs for birth-death models with a Norton-Simon treatment kill-rate (Strobl et al. 2020 \cite{Strobl2020a}):  
$$\frac{dS}{dt}(t) = S(t) \left[b(N(t)) (1 - L(t)) - d(N(t))\right]$$  
where $b(N)$ and $d(N)$ are birth- and death-rates in the absence of treatment. 

Another issue is that \ref{VN_model} does not consider mutations from sensitive to resistant cells. This is problematic because one of the theoretical motivations for aggressive treatments is to decrease tumor size in order to limit the number of reproduction events, hence of possible appearance of resistant cells by mutation. Key-contributions to the tumor containment literature analyzed the trade-off between increasing competition (by allowing many sensitive cells to survive) and decreasing the number of mutations from sensitive to resistant cells (Martin et al. 1992\cite{Martin1992a}, Hansen et al. 2017\cite{Hansen2017}) 

The purpose of our work is to generalize the results of Viossat and Noble
to models that encompass Norton and Simon models, and, at least to a
certain extent, allow for mutations from sensitive to resistant cells. 
Mathematically, this is achieved by formulating the model in terms of
absolute growth-rates and, more importantly, 
by replacing a direct analysis of the evolution through time of the number
of sensitive and resistant cells, $S(t)$ and $R(t)$, by an analysis of the
induced trajectory in what we call the $R-N$ plane, where $N = S+ R$
describes the total tumor size. These trajectories describe the evolution
of the total size $N$ of the tumor as a function of the 
size $R$ of the resistant population. This turns out to be an efficient
technique, allowing to generalize essentially all results of Viossat and
Noble, including the optimality or near-optimality of containment treatments. 

The remainder of this article is organized as follows: the model is
described in the next section. Results are presented in Section
\ref{sec:res}, proved in Section \ref{sec:proofs} and discussed in Section
\ref{sec:disc}. The Appendix elaborates on the extent to which our model allows for mutations from sensitive to resistant cells, and derives the comparison principle on which our results are based. 

\section{Model}
\label{sec:model}
We consider a model with two types of tumors cells: sensitive to treatment, and fully resistant. Their growth is described by differential equations of the form: 
\begin{equation}
\begin{split}
\frac{dS}{dt}(t) & =  \phi_S(S(t), R(t), L(t)), \quad S(0)= S_0 \geq 0\\
\frac{dR}{dt}(t) & =  \phi_R(S(t), R(t)), \qquad \quad R(0)= R_0 >0
\label{VA_model}
\end{split}
\tag{Model 3}
\end{equation}
where $\phi_S$ and $\phi_r$ are continuously differentiable absolute
growth-rate functions. The quantities $\phi_S(0,R,L)$ and $\phi_R(S,0,L)$
are assumed non-negative to ensure that population sizes cannot become
negative. Let $N(t)= S(t) + R(t)$ and $N_0 = S_0 + R_0$. We make the following assumptions: 
\begin{itemize}
\item The patient dies when tumor size reaches a critical size $N_{crit} > N_0$.
\footnote{This assumption is standard but debatable (Mistry, 2020): this will be the topic of some other work.}
\item The size of an untreated tumor increases: $\phi_S(S, R, 0) + \phi_R(S, R) > 0$ if $N \leq N_{crit}$.
\item The higher the treatment level, the lower the growth-rate of sensitive cells: $\phi_S$ is non-increasing in $L$.  
\item The resistant population keeps growing: $\phi_R(S,R) >0$ whenever $R>0$ and $N \leq N_{crit}$, so that the tumor is incurable if, as we assume, resistant cells are initially present. 
\item If $R \geq R_0$ and $N \leq N_{crit}$, for a given number of resistant cells, the larger the sensitive population, 
the lower the growth-rate of resistant cells: $\phi_R$ is non-increasing in $S$. 
\end{itemize}

This last assumption models competition for resources (space, glucose, oxygen) or some other form of inhibition of resistant cells by sensitive cells (Bondarenko et al, 2021 \cite{Bondarenko2021}). 
It neither forbids nor implies a cost of resistance, i.e., that in the absence of treatment, resistant cells grow slower than sensitive cells. In particular, we do not specify whether resistant cells compete more strongly with sensitive cells or with other resistant cells. 

The difference with Viossat and Noble (2021) \cite{Viossat2021} is two-fold: first, the model is formulated in terms of absolute growth-rates, allowing for mutations from sensitive to resistant cells and back.

Second, we make no assumption on how the growth-rate of sensitive cells depends on the number of resistant cells. In particular, $\phi_S$ is not assumed non-increasing in $R$. This model encompasses many previous models (Silva et al. 2012 \cite{Silva2012}, Carrère, 2017 \cite{Carrere2017}, Bacevic and Noble et al. 2017 \cite{Bacevic2017a}, Hansen et al. 2017 \cite{Hansen2017}, Strobl et al. 2020 \cite{Strobl2020a}), including \ref{MG_model}, its original formulation with mutations (Monro and Gaffney, 2009 \cite{Monro2009}), or explicit birth-death models with or without a Norton and Simon treatment effect (Strobl et al. 2020 \cite{Strobl2020a}).

To analyse \ref{VA_model}, it is useful to rewrite it in the equivalent form:
\begin{equation}
\begin{split}
\frac{dN}{dt}(t) & =  f_N(N(t), R(t), L(t))\\
\frac{dR}{dt}(t) & =  f_R(N(t), R(t))
\label{VA_modelbis}
\end{split}
\tag{Model 4}
\end{equation}
where $f_N(N, R, L) = \phi_S(N-R, R, L) + \phi_R(N-R, R)$ and $f_R(N, R) =  \phi_R(N-R, R)$. 
Our main assumptions are then that, on the domain $R_0 \leq R \leq N \leq N_{crit}$, $f_N$ is non-increasing in $L$, positive if $L=0$, and $f_R$ is positive and non-increasing in $N$. 
We also assume that the treatment level cannot be larger than a constant $L_{max}$ (the treatment level corresponding to the maximal tolerated dose).  Other assumptions are technical: 
\begin{itemize}
\item $f_N$ and $f_R$ are continuously differentiable (on a neighborhood of the relevant domain: $R_0 \leq R \leq N \leq N_{crit}$ and $0 \leq L \leq L_{max}$). 
\item $R(t)$ remains smaller than $N(t)$ (this must be biologically, and follows from our assumption on \ref{VA_model} that $\phi_S(0,R,L)$ is nonnegative).
\item The treatment function $L(\cdot)$ is strongly piecewise continuously differentiable (our vocabulary) in the following sense:  
there exists a positive integer $m$ and times $t_0 = 0 < t_1 < ... < t_m$ such that, on each interval $[t_k, t_{k+1})$, $k \in \{0,..., m-1\}$, and on $[t_m, +\infty)$, $L$ coincides with a continuously differentiable function defined on a neighborhood of this interval.
\end{itemize}
This ensures among other things that, for a given initial condition and treatment, there is a unique solution to \ref{VA_modelbis}. To fix ideas, we assume that the solutions $R(t)$ and $N(t)$ are defined for all times (though they have no clear interpretation once $N(t) > N_{crit}$), and that they remain bounded. Both properties can be ensured by modifying growth-rate functions $f_N$ and $f_R$ on the domain $N > N_{crit}$. 
This is without loss of generality since patients are then assumed already deceased. 

\paragraph{Outcomes and treatments. } We compare the effect of various treatments on the time at which tumor size becomes larger than a given threshold. Depending on this threshold, this may correspond to: 
\begin{itemize}
\item time to progression, defined as the time at which tumor size progresses beyond its initial size $N_0$.\footnote{In the response evaluation criteria in solid tumors (RECIST), progressive disease is defined by a 20\% increase in the sum of the largest diameters (LD) of target lesions, compared to the smallest LD sum recorded since the beginning of treatment. However, comparing to the  smallest LD sum recorded would not be fair to aggressive treatments, and the 20\% margin makes sense in medical practice, to take into account imperfect monitoring and imperfect forecast of treatment effect, but not for our deterministic mathematical model.} 
\item time to treatment failure: the time at which tumor size progresses beyond an hypothetical maximal tolerable tumor size $N_{tol} \geq N_0$, 
after which the life of the patient is considered at risk or side-effects of the disease are too strong.\footnote{The assumption $N_{tol} \geq N_0$ in without loss of generality in the following sense: if the initial size is larger than the maximal tolerable size, then all treatments we consider would treat at $L_{max}$ until tumor size becomes tolerable ($N = N_{tol}$), and we could apply our analysis from that point on.}
\item survival time, 
defined as the time at which tumor size becomes larger than a critical size $N_{crit} \geq N_{tol}$. 
\end{itemize}
Mathematically, results on time to progression and survival time may be obtained through results on time to treatment failure by taking $N_{tol} = N_0$, or $N_{tol} = N_{crit}$, respectively. For this reason, we focus on time to treatment failure. 
 
We consider the following treatments: 

$\bullet$ Constant dose treatments, 
including \emph{No treatment} (noTreat): $L(t) = 0$, and \emph{Maximal Tolerated Dose} (MTD): $L(t) = L_{max}$ throughout. 

$\bullet$ \emph{Delayed MTD} (del-MTD): do not treat until $N=N_{tol}$ for the first time, then treat at $L_{max}$ for ever. 

$\bullet$ \emph{Containment at $N_{tol}$} (Cont): do not treat until $N = N_{tol}$ and then stabilize tumor size at $N_{tol}$, as long as possible with a treatment level $L(t) \leq L_{max}$. 
Finally, treat at $L_{max}$ when $N > N_{tol}$. Formally, during the stabilization phase, the treatment level is chosen so that $dN/dt = 0$ (e.g., $L(t)= N(t)/S(t)$ in \ref{MG_model}). Containment treatments are illustrated in Fig. \ref{fig:3}, see also Fig. 1 of Viossat and Noble.\footnote{If, after crossing $N_{tol}$, tumor size comes back to $N_{tol}$, then the containment treatment stabilizes tumor size at $N_{tol}$ again, 
as long as possible. Similar remarks apply to intermittent containment or other variants of containment.}
    \begin{figure}[htbp]
    \centering
     \includegraphics[scale=0.31]{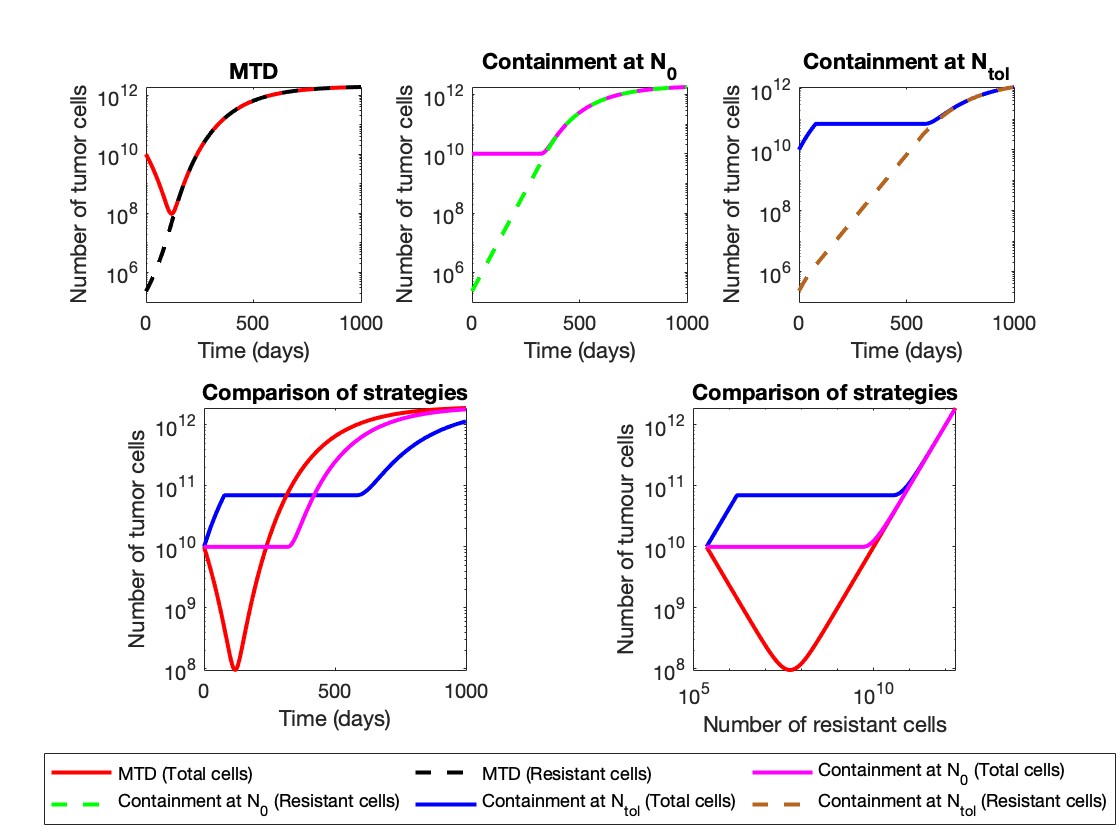}
     \caption{Number of resistant cells and number of tumor cells for different treatments. Top row: number of resistant cells and number of tumor cells as a function of time under MTD (left), Containment at the initial size $N_0$ (center), and Containment at the maximal tolerable size $N_{tol}$ (right). Bottom-row:  number of tumor cells under these three treatments as a function of time (left) or as a function of the number of resistant cells (right).}
     \label{fig:3}
 \end{figure}
$\bullet$ \emph{Intermittent containment} (Int), as in the prostate cancer clinical trial of Zhang et al. (2017): do not treat until $N = N_{tol}$, then treat at $L_{max}$ until $N=N_{min} < N_{tol}$, then interrupt treatment until $N=N_{tol}$, and iterate as long as possible. Finally, treat at $L_{max}$ when $N > N_{tol}$. This is illustrated by Fig. \ref{fig:2}.\footnote{The above description is to fix ideas: our results are still valid for any other way of maintaining tumor size between $N_{min}$ and $N_{tol}$, as long as this may be done with a dose $L(t) \leq L_{max}$.}
    \begin{figure}[htbp]
    \centering
     \includegraphics[scale=0.2]{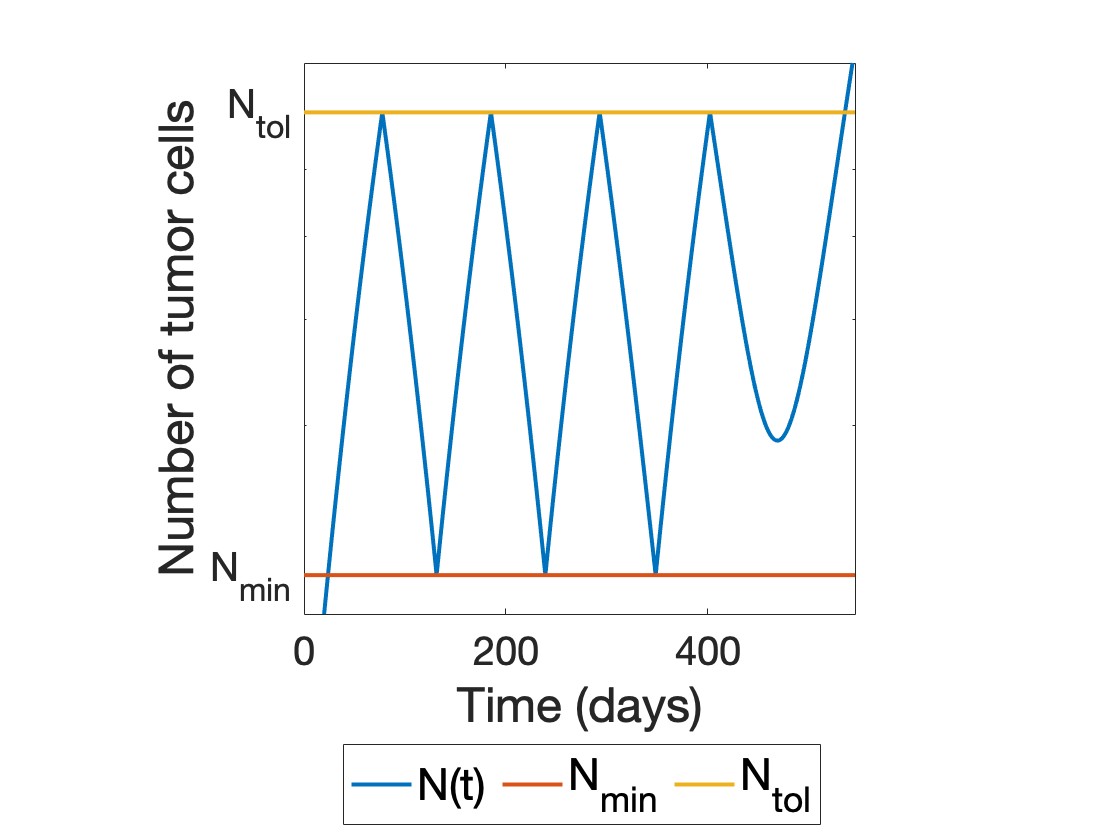}
     \caption{Total size under Intermittent treatment.}
     \label{fig:2}
 \end{figure}
$\bullet$ An arbitrary treatment, called the \emph{alternative treatment} (alt): we only assume that $0 \leq L(t) \leq L_{max}$ for all $t$ and $L(t) = L_{max}$ if $N>N_{tol}$.

The times to treatment failure under these treatments will be denoted by $t_{noTreat}$, $t_{MTD}$, $t_{delMTD}$, $t_{Cont}$, $t_{Int}$, 
and $t_{alt}$, respectively. 

Following (Martin et al. 2012 \cite{Martin1992}, Hansen et al. 2017 \cite{Hansen2017}, Viossat and Noble 2021 \cite{Viossat2021}), we also consider idealized treatments, which assume that the sensitive population may be reduced arbitrarily quickly.  
These treatments are not realistic but are useful theoretical references. \emph{Ideal MTD}(idMTD) eliminates all sensitive cells instantaneously at the beginning of treatment. 
\emph{Delayed ideal MTD} (del-idMTD) lets tumor grow to $N_{tol}$ and then eliminates all sensitive cells. \emph{Ideal containment} (idCont) lets tumor size grow to $N_{tol}$, 
and then stabilizes it as long as some sensitive cells remain. Finally, \emph{Ideal intermittent containment} (idInt) lets tumor size grow to $N_{tol}$ and then maintains it between 
$N_{min} \leq N_{tol}$ and $N_{tol}$ as long as some sensitive cells remain.\footnote{Viossat and Noble assumed to fixed ideas and for simulations that, each time tumor size 
reaches $N_{tol}$, it drops instantaneously to $N_{min}$, or to $R(t)$ if $R(t) > N_{min}$, but this is not needed.} 

Under these idealized treatments, treatment fails (i.e. tumor size progresses beyond $N_{tol}$) when the resistant population reaches size $N_{tol}$. Sensitive cells have then been fully eliminated. Times to treatment failure are denoted by $t_{idMTD}$, $t_{del-idMTD}$, $t_{idCont}$, 
and $t_{idInt}$, respectively.\footnote{When comparing idealized treatments to an alternative treatment, for the comparison to be fair, we do not restrict treatment level under the alternative treatment either, and allow it to eliminate sensitive cells arbitrarily quickly.} 

To make our life easy, we assume that all treatments we consider may be implemented through a piecewise continuously differentiable treatment level function $L(t)$ (up to possible downward jumps in the sensitive population for idealized treatments), instead of deriving this result from the implicit function theorem and appropriate regularity assumptions. 

\section{Results}
\label{sec:res}
We show that, up to natural additional assumptions for comparisons of sensitive cell populations, all results of Viossat and Noble on \ref{VN_model} still hold on \ref{VA_model} (or equivalently \ref{VA_modelbis}), in spite of our less restrictive assumptions. 
The results are described below and proved in the next section. 

The key point is that if treatment level is never larger than a given constant for treatment 1, and never smaller than the same constant for treatment 2, then the resistant population is no larger under treatment 1 than under treatment 2. 
\begin{prop} \label{prop:key}
Consider solutions of \ref{VA_modelbis} associated to two treatments $L_1(t)$ and $L_2(t)$.\footnote{Unless mentioned otherwise, when comparing two treatments, we assume the same initial conditions: $R_1(0) = R_2(0)$ and $N_1(0) = N_2(0)$.} 
If there exists a constant $\bar{L}$ such that for all $t \geq 0$, $L_1(t) \leq \bar{L} \leq L_2(t)$, then $R_1(t) \leq R_2(t)$ for all $t \geq 0$.  
\end{prop}

It follows that for constant dose treatments, lowering the dose or delaying treatment leads to a lower resistant population:
\begin{prop}\label{prop:constantdelay} (constant dose treatments)\\
a) Consider two constant dose treatments $L_1(t) = L_1$ and $L_2(t) = L_2$. If $L_2 \geq L_1$, then $R_1(t) \leq R_2(t)$ for all $t \geq 0$.\footnote{Exceptionally, we assume that even when $N_i >N_{tol}$, the dose stays equal to $L_i$ and is not increased to $L_{max}$. A similar remark applies to b)} 

b) Assume that $L_1(t) = L > 0$ for all $t \geq 0$, while $L_2(t) = 0$ until $N = N_{start} \geq N_0$, and then $L_2(t) =L$. Then $R_1(t) \leq R_2(t)$ for all $t \geq 0$. 
\end{prop}
 
Proposition \ref{prop:key} also implies that not treating minimizes the resistant population while MTD maximizes it: 
\begin{prop}\label{prop:alt0} (MTD maximisez resistance)\\ 
For all $t \geq 0$, $R_{noTreat}(t) \leq R_{alt}(t) \leq R_{MTD}(t)$.
\end{prop}
Of course, not treating is typically not an option, as the number of sensitive cells would explode, but containment is. One of our main results is that containment minimizes the resistant population among all treatments treating at $L_{max}$ after failing. 
\begin{prop}\label{prop:alt1} (containment minimizes resistance)\\ For all $t \geq 0$, $R_{Cont}(t)  \leq R_{alt}(t)$
\end{prop}
It follows that $N_{Cont}  \leq N_{alt} + (S_{Cont} - S_{alt}).$ Thus, assuming that the tumor is eventually mostly resistant under the containment treatment, tumor size should eventually be smaller, or at least not substantially larger under the containment treatment than under any alternative one. This suggests that, under our assumptions, among treatments that treat at $L_{max}$ when $N> N_{tol}$, containment should be close to maximizing survival time. Similarly, the fact that the resistant population is larger under MTD that under any alternative treatment suggests that most alternative treatments should eventually lead to a lower tumor size and a longer survival time than MTD. 

More precise statements may be made for idealized versions of containment and MTD: 
ideal containment maximizes time to treatment failure, while ideal MTD minimizes it among 
all treatments eliminating sensitive cells before failing. Moreover, ideal containment eventually 
leads to a lower tumor size and ideal MTD to a larger tumor size than any such alternative treatment. 
\begin{prop}\label{prop:alt2} (comparison with ideal MTD and ideal containment)\\
a) $t_{alt} \leq t_{idCont}$.\\
b) Consider an alternative treatment eliminating sensitive cells before failing, that is, such that $S_{alt}(t_{alt}) = 0$. Then: 

b1) $t_{alt} \geq t_{idMTD}$; 

b2) for all $t \geq 0$, $R_{idCont}(t)  \leq R_{alt}(t) \leq R_{idMTD}(t)$; 

b3) for all $t \geq t_{alt}$, $N_{idCont}(t) \leq N_{alt}(t) \leq N_{idMTD}(t)$.  

In particular, survival time is larger with ideal containment and lower with ideal MTD than with any alternative treatment such that $S_{alt}(t_{alt}) = 0$. 
\end{prop}

The next result shows that intermittent containment between $N_{min}$ and $N_{tol} > N_{min}$ leads to outcomes that are intermediate between those of containment at the larger threshold $N_{tol}$ and those of containment at the lower threshold $N_{min}$ (ContNmin). The latter lets tumor size grow until $N = N_{min}$ (or treats at $L_{max}$ until $N=N_{min}$ if $N_0 > N_{min}$), and then stabilizes tumor size at $N_{min}$ as long as possible with a treatment level $L(t) \leq L_{max}$. In the idealized form, ideal containment at $N_{min}$, tumor size is stabilized at $N_{min}$ as long as some sensitive cells remain (and initially instantly reduced to the maximum of $N_{min}$ and $R_0$, if $N_{min} > N_0$). 
\begin{prop} (dose modulation versus treatment vacation)\\
\label{prop:int} 
a) For all $t \geq 0$, $R_{Cont}(t) \leq R_{Int}(t) \leq R_{ContNmin}(t)$, and similarly,  $R_{idCont}(t) \leq R_{idInt}(t) \leq R_{idContNmin}(t)$. \\
b) $t_{idContNmin} \leq t_{idInt} \leq t_{idCont}$\\
c) For all $t \geq t_{idInt}$, $N_{idCont}(t) \leq N_{idInt}(t) \leq N_{idContNmin}(t)$. 
\end{prop}
This result suggests that, if the lower threshold $N_{min}$ is close to the larger threshold $N_{tol}$, there should be little difference between outcomes of containment and intermittent containment, that is, between a continuous low dose treatment based on dose modulation and an intermittent high dose treatment based on treatment vacation. Of course, this disregards many possible differences between these two approaches. For instance, dose-modulation might lead to a more regular vascularization of the tumor, which might be key for an efficient drug delivery (Enriquez-Navas et al., 2016 \cite{Enriquez-Navas2016}).

We now compare all reference treatments. 

\begin{prop}
\label{prop:ref} (comparison between all reference treatments)\\
a) For all $t \geq 0$:\\
a1) $R_{noTreat}(t) \leq R_{Cont}(t) \leq R_{Int}(t) \leq R_{del-MTD}(t) \leq R_{MTD}(t) \leq R_{idMTD}(t)$\\
and a2) $R_{noTreat}(t) \leq R_{idCont}(t) \leq R_{idInt}(t) \leq R_{del-idMTD}(t) \leq R_{idMTD}(t)$\\
b) $t_{idMTD} \leq t_{del-idMTD} \leq t_{idInt} \leq t_{idCont}$\\
c) For all $t \geq t_{idCont}$, $N_{idCont}(t) \leq N_{idInt}(t) \leq N_{del-idMTD}(t) \leq N_{idMTD}(t)$. 
\end{prop}
Sensitive population sizes may also be compared under two mild additional assumptions: 

(A1) Not treating maximizes the sensitive population. 

(A2) The sensitive population decreases if the tumor is treated at $L_{max}$. 

These assumptions hold for \ref{MG_model}, assuming $L_{max} \geq 1$, and for most models we are aware of. 
They lead to the same comparison for sensitive population sizes as in Viossat and Noble, that is, the opposite as for resistant population sizes.\footnote{For \ref{MG_model}, (A2) holds obviously if  $L_{max} \geq 1$. To see that (A1) holds, 
note that $dN/dt =  \rho \ln(K/N) N - \rho \ln(K/N) L S \geq \rho \ln(K/N) N$ with equality for $L =0$. 
The comparison principle thus implies that not treating maximizes tumor size. Since not treating also minimizes the resistant population size (Proposition \ref{prop:alt0}), it follows that it maximizes the sensitive population size.} 
\begin{prop} 
\label{prop:S} (comparison of sensitive populations)\\
Assume that (A1) and (A2) hold. Then for all $t \geq 0$:\footnote{The only inequality that uses (A1) is $S_{alt} \leq S_{Cont}$.}\\ 
a) $S_{idMTD}(t) \leq S_{MTD}(t) \leq S_{alt}(t) \leq S_{Cont}(t) \leq S_{noTreat}(t)$\\
b) $S_{ContNmin}(t) \leq S_{Int}(t) \leq S_{Cont}(t)$  and $S_{del-MTD}(t) \leq S_{Int}(t)$\\
c)  $S_{idContNmin}(t) \leq S_{idInt}(t) \leq S_{idCont}(t)$\\
d) $S_{idMTD}(t) \leq S_{del-idMTD}(t) \leq S_{idInt}(t) \leq S_{idCont}(t) \leq S_{noTreat}(t)$
\end{prop}

\section{Proofs}
\label{sec:proofs}
Viossat and Noble's proofs build on their Proposition 1, which gives conditions allowing them to compare the resistant populations or the sensitive populations under two different treatments. The following part of this result is still true in our framework, with the same proof: 
\begin{lemma}\label{lem:VN}
Let $0 \leq t_0 \leq t_1$. Consider two solutions $(S_1, R_1)$ and $(S_2, R_2)$ of \ref{VA_model}, associated to treatment functions $L_1$ and $L_2$, respectively. Assume that: i) $R_1(t_0) \leq R_2(t_0)$, and ii) $S_1(t_0) \geq S_2(t_0)$. If: iiia) $S_1(t) \geq S_2(t)$ on $[t_0, t_1]$, or: iiib) $N_1(t) \geq N_2(t)$ on $[t_0, t_1]$, then $R_1(t) \leq R_2(t)$ and $S_1(t) \geq S_2(t)$ on $[t_0, t_1]$.\end{lemma}
What is no longer true is that the same conclusions hold if iiia) or iiib) is replaced by iiic): $L_1(t) \leq L_2(t)$ for all $t$. For instance, in \ref{MG_model}, if $L_1(t_0) = L_2(t_0) >1$, $R_1(t_0) < R_2(t_0)$ and $S_1(t_0) = S_2(t_0)$, then $S_2$ becomes immediately larger than $S_1$. This will slow down the growth of the resistant population under treatment 2. 
Thus, conceivably, $R_2$ could later on become smaller than $R_1$. 

We thus use a new proof technique. Instead of studying directly the evolution of the resistant population $R$, 
the sensitive population $S$, or the total tumor size $N$ as a function of time, we first study, and compare across treatments, 
the evolution of tumor size $N$ \emph{as a function of the number of resistant cells}. In other words, we compare trajectories 
in the $R-N$ plane, that is, the sets of points $(R(t), N(t))$ for all $t\geq 0$. 

To be more formal, fix a treatment $L$, and let $R^{\infty} = \lim_{t \to +\infty} R(t)$. 
Since the resistant population increases continuously, for any $r \in [R_0, R^{\infty})$, there exists a unique time $t(r)$ at which the 
resistant population has size $r$, that it, $R(t(r)) = r$. Denote by $\tilde{S}(r)$, $\tilde{N}(r) = \tilde{S}(r) + r$, and $\tilde{L}(r)$, the number of sensitive cells, the total number of tumor cells, and the treatment level at time $t(r)$, that is, when the resistant population reaches size $r$. All these functions may be shown to be piecewise continuously differentiable, and $\tilde{S}$ and $\tilde{N}$ are also continuous. The graph of function $\tilde{N}$ coincides with the trajectory of the solution in the $R-N$ plane. It may be analyzed by noting that function $\tilde{N}$ satisfies the differential equation: 
\[\frac{d \tilde{N}}{d r}  = G(\tilde{N}, r) \mbox{ where } 
G(\tilde{N}, r) = \frac{f_N(\tilde{N}, r, \tilde{L}(r))}{f_R(\tilde{N},r)}\]
Trajectories in the $R-N$ plane, and their connections to the evolution of tumor size and of the resistant population as a function of time 
are illustrated in Figure \ref{fig:3}.

\subsection{Key lemmata}

Our first result shows that if, for any resistant population level $r$, tumor size is larger under treatment 1 than under treatment 2, then at any time $t$, the resistant population is smaller under treatment 1 than under treatment 2. The intuition is the following: at the time $t_i(r)$ when the resistant population reaches size $r$ under treatment $i$, the speed at which the resistant population increases is given by: 
\[\frac{dR_i}{dt}(t_i(r)) = f_R(N_i(t_i(r)), R_i(t_i(r)) = f_R(\tilde{N}_i(r), r)\] 
Since $f_R$ is 
non-increasing in $N$, it follows that if $N_1(r) \geq N_2(r)$, the resistant population will increase quicker from $r$ to $r+dr$ under treatment 2 
than under treatment 1 ($dr$ is a small positive increment). If this holds for all resistant population sizes $r$, then $R_2$ will remain 
no-smaller than $R_1$ at all times $t \geq 0$. 

\begin{lemma}\label{L1L2}
Let $L_1(t)$ and $L_2(t)$ be two different treatments. Consider solutions $(N_1, R_1)$ and $(N_2, R_2)$ of \ref{VA_modelbis} 
associated to these treatments such that $R_1(0) = R_2(0) = R_0$. If  $\tilde{N}_1(r)\geq  \tilde{N}_2(r)$ for all $r$ in 
$[R_0,\min\{R^{\infty}_1,R^{\infty}_2\})$, then $R_1^{\infty} \leq R_2^{\infty}$ 
and $R_1(t)\leq  R_{2}(t)$ for all $t\geq 0$. 

Moreover, if on an interval $[r_1,r_2]$,  $\tilde{S}_1(r)$ is non-increasing or $\tilde{S}_2(r)$ is non-increasing, then $S_1(t) \geq  S_{2}(t)$ 
for all $t$ in $[t_1(r_1), t_2(r_2)]$.
\end{lemma}
\begin{proof}
Consider a time $t\geq 0$ such that $R_1(t) < R_2^{\infty}$, so that $\tilde{N}_2(R_1(t))$ is well defined. Since $N(t)= \tilde{N}(R(t))$, $\tilde{N}_1 \geq \tilde{N}_2$ and $f_N$ is non-increasing in $N$, we obtain: 
\begin{align*}
    \frac{dR_1}{dt}(t) = f_R(N_1(t), R_1(t)) = f_R(\tilde{N}_1(R_1(t)), R_1(t)) \leq& f_R(\tilde{N}_2(R_1(t)), R_1(t))\\
    =:&F(R_1(t))
\end{align*}
while 
\[\frac{dR_2}{dt}(t) = f_R(\tilde{N}_2(R_2(t)), R_2(t))=F(R_2(t)).\]
Since $R_1(0) = R_2(0)$, the comparison principle (Proposition \ref{prop:cp}, item b), in Appendix \ref{app:cp}) implies that for all times $t \geq 0$ 
such that $R_1(t) < R_2^{\infty}$, we have $R_1(t) \leq R_2(t)$ (See Figure \ref{fig:4}).

\begin{figure}[htbp]
    \centering
    \includegraphics[scale=0.25]{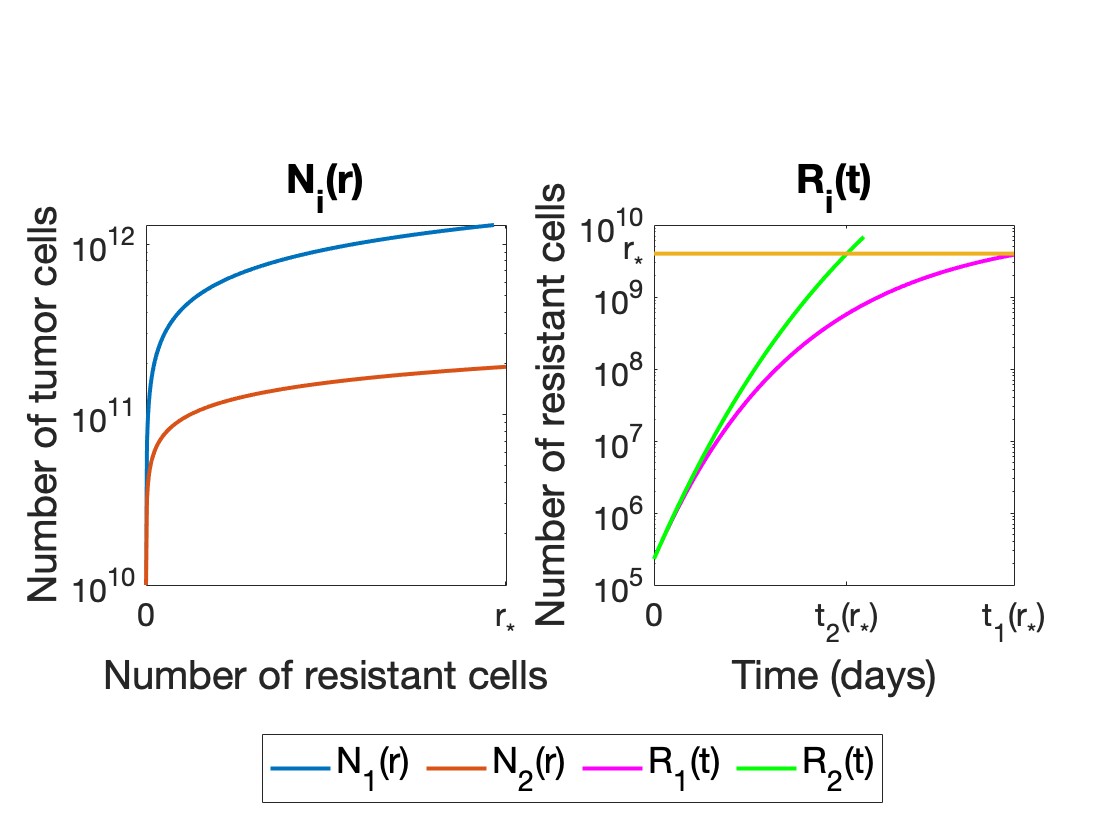}
    \caption{Using two different constant treatment levels on the Gompertzian model as an example, this figure illustrates that the relation $\tilde{N}_1\geq  \tilde{N}_2$ over the interval $[R_0,r_*)$, $r_*=4\times 10^9$ (Left panel), translates into $R_1(t)\leq  R_2(t)$ over the interval $[0,\min\{t_1(r_*),t_2(r_*)\}]$ (Right panel). The same behavior will be observed for an arbitrary choice of $r_*$.  }
     \label{fig:4}
\end{figure}

We now show that the inequality $R_1(t) < R_2^{\infty}$, hence the conclusion 
$R_1(t) \leq R_2(t)$, holds at all times $t \geq 0$. Indeed, otherwise there is a first time $t^* \geq 0$ such that $R_1(t^*) = R_2^{\infty}$, and 
$t^*>0$. Since $R_i$ is increasing, it follows that on $[0, t^*)$, $R_1(t) \leq R_2(t) \leq R_2(t^*) < R_2^{\infty}$. By continuity of $R_1$, this implies 
that  $R_1(t^*) \leq R_2(t^*) < R_2^{\infty}$, a contradiction.

We now prove the result on sensitive cells. Assume that on $[r_1, r_2]$, $\tilde{S}_1(r)$ is non-increasing (which implicitly requires $r_2 < R_1^{\infty}$, so that $\tilde{S}_1(r)$ is well-defined on $[r_1, r_2]$). For all $t$ in $[t_1(r_1), t_2(r_2)]$, $r_1 \leq R_1(t) \leq R_2(t) \leq r_2$. Moreover, the assumption $\tilde{N}_1(r) \geq \tilde{N}_2(r)$ is equivalent to $\tilde{S}_1(r) \geq \tilde{S}_2(r)$. Thus we obtain: 
\[S_1(t) = \tilde{S}_1(R_1(t)) \geq \tilde{S}_1(R_2(t)) \geq \tilde{S}_2(R_2(t)) = S_2(t).\] 
The first inequality follows from the fact that $\tilde{S}_1$ is non-increasing on $[r_1, r_2]$, the second from the fact that $\tilde{S}_1 \geq \tilde{S}_2$. 
If it is $\tilde{S}_2$ which is non-increasing (which implicitly requires $r_2 < R_2^{\infty}$, so that $\tilde{S}_2$ is well-defined on $[r_1, r_2]$), then: 
\[S_1(t) = \tilde{S}_1(R_1(t)) \geq \tilde{S}_2(R_1(t)) \geq \tilde{S}_2(R_2(t)) = S_2(t).\]  
The first inequality follows from the fact that $\tilde{S}_1 \geq \tilde{S}_2$, the second from the fact that $\tilde{S}_2$ is non-increasing on $[r_1, r_2]$.
\end{proof}

Assume now that for any resistant population level $r$, the treatment level when the resistant population reaches size $r$ is lower for treatment 1 than for treatment 2. Our second result shows that the tumor size when the resistant population reaches size $r$ is then always larger for treatment 1 than for treatment 2. By the previous lemma, this implies that the resistant population is always smaller under treatment 1 than under treatment 2. 
 \begin{lemma}
 \label{lem:CompC}
Let $L_1(t)$ and $L_2(t)$ be two different treatments such that
\[
L_1(t)\leq  \overline{L}\leq  L_2(t),
\]
for a certain positive number $\overline{L}$, or more generally such that $\tilde{L}_1(r) \leq \tilde{L}_2(r)$ for all $r$ in $[R_0, R^*)$ where $R^* = \min\{R_1^{\infty},R_2^{\infty}\}$. 
Consider solutions of \ref{VA_modelbis} associated to these treatments such that $R_1(0)=R_2(0)$ and $N_1(0) \geq N_2(0)$. Then $\tilde{N}_1(r)\geq  \tilde{N}_2(r)$ for all $r$ in $[R_0, R^*)$. Therefore by Lemma \ref{L1L2}, $R_1^{\infty} \leq R_2^{\infty}$ and $R_1(t) \leq R_2(t)$ for all $t \geq 0$. 
\end{lemma}
\begin{proof}
Since $f_N$ is non-increasing in $L$, and for all $r$ in $[R_0,R^*)$, $\tilde{L}_1(r)\leq  \tilde{L}_2(r)$, we get: 
\[
\frac{d\tilde{N}_1}{dr} =\frac{f_N(\tilde{N}_1, r, \tilde{L}_1(r))}{f_R(\tilde{N}_1, r)}
 \geq  \frac{f_N(\tilde{N}_1, r, \tilde{L}_2(r))}{f_R(\tilde{N}_1, r)} =: G_2(\tilde{N}_1,r)
\]
while $\displaystyle \frac{d\tilde{N}_2}{dr} =G_2(\tilde{N}_2, r)$. Moreover, $\tilde{N}_{2}(R_0) = N_2(0) \leq  N_1(0) = \tilde{N}_1(R_0)$. Therefore, by the comparison principle (Proposition \ref{prop:cp}, item a), in Appendix \ref{app:cp}), $\tilde{N}_1(r)\geq  \tilde{N}_2(r)$ for all $r$ in $[R_0,R^*)$. Then apply Lemma \ref{L1L2}.
\end{proof}

\subsection{Proof of propositions \ref{prop:key} to \ref{prop:S}}

Proposition \ref{prop:key} follows from Lemma \ref{lem:CompC}, and Propositions \ref{prop:constantdelay} and \ref{prop:alt0} from Proposition \ref{prop:key}. \\ 

\noindent \emph{Proof of Proposition \ref{prop:alt1}.} For later purposes, let us prove a more general result: For all $t \geq 0$, \begin{equation}
    \label{eq:bonus}
R_{noTreat}(t) \leq R_{Cont}(t)  \leq R_{alt}(t) \leq R_{MTD}(t) \leq R_{idMTD}(t).
\end{equation}
This follows from Lemma \ref{L1L2} and the fact that, whenever these comparisons make sense: 
\begin{equation}
\label{eq:tildeN} 
\tilde{N}_{idMTD}(r) \leq \tilde{N}_{MTD}(r) \leq \tilde{N}_{alt}(r) \leq \tilde{N}_{Cont}(r) \leq \tilde{N}_{noTreat}(r)
\end{equation}
To prove \eqref{eq:tildeN}, note that for any alternative treatment, 
$\tilde{N}_{idMTD}(r) = r \leq \tilde{N}_{alt}(r)$, in particular, 
$\tilde{N}_{idMTD}(r) \leq \tilde{N}_{MTD}(r)$, 
and by Lemma \ref{lem:CompC} with $\bar{L} = 0$, $\tilde{N}_{alt}(r) \leq \tilde{N}_{noTreat}(r)$, 
in particular $\tilde{N}_{Cont}(r) \leq \tilde{N}_{noTreat}(r)$. Moreover, under the constraint $L_{alt}(t) \leq L_{max}$, 
it follows from Lemma \ref{lem:CompC} with $\bar{L} = L_{max}$ that $\tilde{N}_{MTD}(r) \leq \tilde{N}_{alt}(r)$. 

It remains to prove that $\tilde{N}_{alt}(r) \leq \tilde{N}_{Cont}(r)$ for all $r \in [R_0, R^*)$, where 
$R^* = \min \{R_{alt}^{\infty}, R_{Cont}^{\infty}\}$. The notation we introduce is illustrated 
in Fig.~\ref{fig:5}. 
Let $r_1 = \min\{r\geq R_0, \tilde{N}_{Cont}(r) = N_{tol}\}$. 
When $r \leq r_1$,  $\tilde{N}_{Cont}(r) = \tilde{N}_{noTreat}(r) \geq \tilde{N}_{alt}(r)$ as explained above. 
Moreover, for all $r \geq r_1$, $\tilde{N}_{Cont}(r) \geq N_{tol}$. Thus, assuming by contradiction that 
there exists $r_2 \geq r_1$ such that $\tilde{N}_{Cont}(r_2) < \tilde{N}_{alt}(r_2)$, it follows that $\tilde{N}_{alt}(r_2) > N_{tol}$. 
Let 
\[r_{max} = \max\{r \leq r_2, \tilde{N}_{alt}(r) \leq  N_{tol}\}.\]
Note that since $\tilde{N}_{alt}(r_1) < \tilde{N}_{Cont}(r_1) = N_{tol}$, we must have $r_{max} \geq r_1$. 
Therefore, $\tilde{N}_{alt}(r_{max}) =N_{tol} \leq \tilde{N}_{Cont}(r_{max})$. Moreover,  
on $(r_{max},r_2)$, $\tilde{N}_{alt}(r) > N_{tol}$, 
hence $\tilde{L}_{alt}(r) = L_{max} \geq \tilde{L}_{Cont}(r)$. 
By a variant of Lemma \ref{lem:CompC} (comparing treatments starting when the initial resistant population size is $r_{max}$ rather than $R_0$), 
it follows that $\tilde{N}_{alt}(r_2) \leq \tilde{N}_{Cont}(r_2)$, a contradiction.

\begin{figure}[htbp]
    \centering
     \includegraphics[scale=0.2]{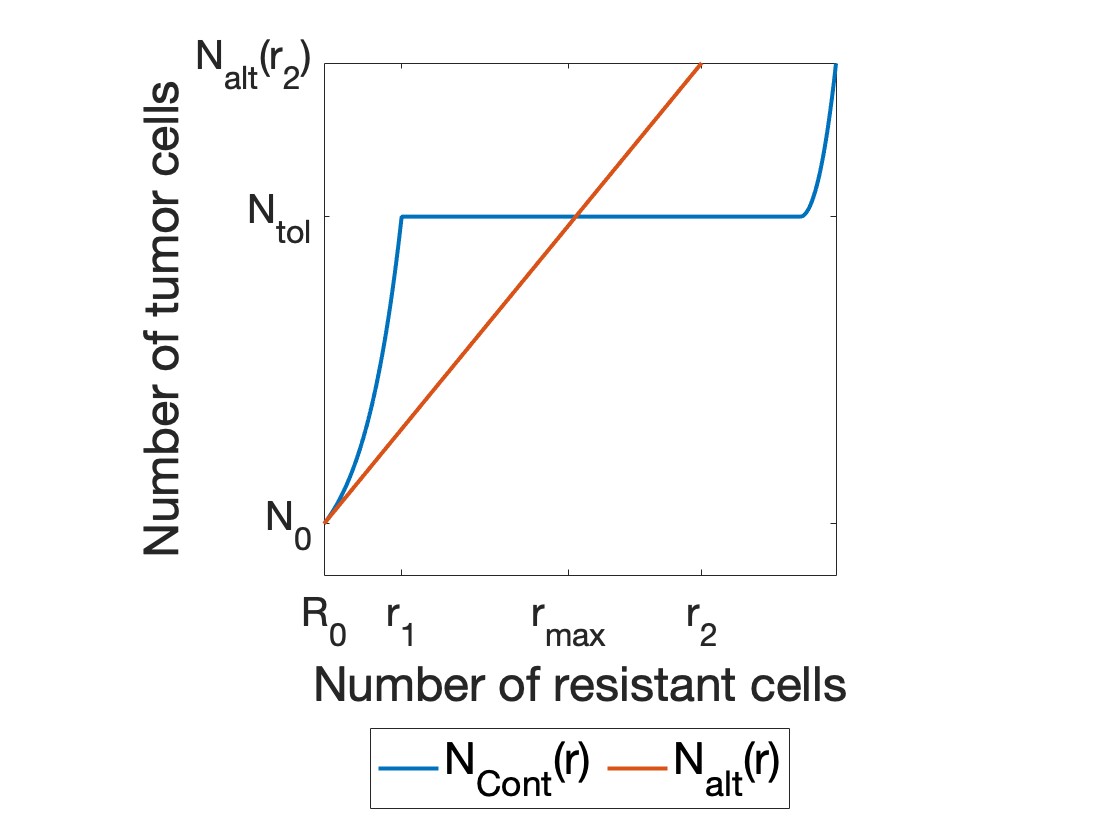}
     \caption{Comparison of $N_{Cont}(r)$ with an hypothetical curve which satisfies $N_{alt}(r_2)>N_{Cont}(r_2)$ for some point $r_2\geq  r_1$.}
     \label{fig:5}
\end{figure}

\noindent \emph{Proof of Proposition \ref{prop:alt2}.} 
Proof of a): Let $t_1 = \min\{t \geq 0, N_{Cont}(t) = N_{tol}\}$. For $t \leq t_1$,  
$R_{idCont}(t) = R_{noTreat}(t) \leq R_{alt}(t)$ (the inequality follows from Proposition \ref{prop:alt0}).  
If $t_{alt} \geq t_1$ then, as in Viossat and Noble, on $[t_1, t_{alt}]$, $N_{idCont}(t) \geq N_{tol} \geq N_{alt}(t)$ so $R_{idCont}(t) \leq R_{alt}(t)$ by Lemma \ref{lem:VN}. 
Thus, $R_{idCont}(t_{alt}) \leq R_{alt}(t_{alt}) \leq N_{alt}(t_{alt})= N_{tol}$. It follows that $t_{idCont} \geq t_{alt}$, since ideal containment fails when $R_{idCont} = N_{tol}$.

 Proof of b): assume now that  $S_{alt}(t_{alt})=0$ (which only makes sense for idealized alternative treatments). 
Then $R_{alt}(t_{alt}) = N_{alt}(t_{alt}) = N_{tol}$. Thus, as in Viossat and Noble: 
 \[N_{tol} = N_{alt}(t_{alt}) = R_{alt}(t_{alt}) \leq R_{idMTD}(t_{alt}),\] so $t_{idMTD} \leq t_{alt}$. This proves b1).  
 
 Let us prove the remaining results on ideal MTD. 
 The inequality $R_{alt} \leq R_{idMTD}$ was shown in the proof of Proposition \ref{prop:alt1} (see Eq. \ref{eq:bonus}). Moreover, on $[t_{idMTD}, t_{alt}]$, $N_{idMTD} = R_{idMTD} \geq N_{tol} \geq N_{alt}$, 
 and for $t \geq t_{alt}$, $N_{idMTD} = R_{idMTD} \geq R_{alt} = N_{alt}$. This proves parts of b2) and b3). 
 
We now prove the results on ideal containment. 
On $[t_{alt}, t_{idCont}]$, 
\[R_{idCont}(t) \leq N_{idCont}(t) \leq N_{tol} = N_{alt}(t_{alt}) = R_{alt}(t_{alt}) \leq R_{alt}(t) = N_{alt}(t)\]
where the last inequality comes from the fact that for all $t \geq t_{alt}$, $S_{alt}(t) =0$. 
Moreover, after treatment failure, $R_{alt}$ and $R_{idCont}$ both satisfy the autonomous equation $dR/dt = f_R(0, R)$. 
By invariance of solutions of autonomous equations through translation in time, this implies that for all $t \geq t_{idCont}$, 
$R_{idCont}(t) = R_{alt}(t - [t_{idCont} - t_{alt}]) \leq R_{alt}$. For $t \leq t_{idCont}$, the inequality $R_{idCont}(t) \leq R_{alt}(t)$ 
was derived in the proof of a). Therefore, $R_{idCont}(t) \leq R_{alt}(t)$ for all $t \geq 0$. 
Finally, for all $t$ in $[t_{alt}, t_{idCont}]$, $N_{idCont}(t) \leq N_{tol} = R_{alt}(t_{alt}) \leq R_{alt}(t) = N_{alt}(t)$, while for all $t \geq t_{idCont}$, $N_{idCont}(t) = R_{idCont}(t) \leq R_{alt}(t) = N_{alt}(t)$. This completes the proof.\\   

\noindent \emph{Proof of Proposition \ref{prop:int}}. Proof of a): The inequalities 
$R_{Cont} \leq R_{Int}$ and $R_{idCont} \leq R_{idInt}$ follow from the proof of 
Proposition \ref{prop:alt1} (see Eq. \eqref{eq:bonus}) and from Proposition \ref{prop:alt2}. 
The fact that $R_{Int} \leq R_{ContNmin}$ follows from Lemma \ref{L1L2} and the fact that, as shown below: for all $r$, $\tilde{N}_{ContNmin}(r) \leq \tilde{N}_{Int}(r)$. 
To prove this, note that for $r \leq r_{min} := \min\{r \geq R_0, \tilde{N}_{noTreat}(r) = N_{tol}\}$, both treatments coincide so $\tilde{N}_{ContNmin}(r) = \tilde{N}_{Int}(r)$. 
For $r \geq r_{min}$, the argument is as in the proof of $\tilde{N}_{Cont}(r) \geq \tilde{N}_{alt}(r)$ for $r \geq r_1$ in Proposition \ref{prop:alt1}. 
Similarly, it is easily seen that $\tilde{N}_{idContNmin}(r) \leq \tilde{N}_{idInt}(r)$ for all $r$, so $R_{idInt} \leq R_{idContNmin}$ by Lemma \ref{L1L2}. 

Proof of b): $N_{tol} = N_{idInt}(t_{idInt}) = R_{idInt}(t_{idInt}) \leq R_{idContNmin}(t_{idInt})$ by a), hence $t_{idContNmin} \leq t_{idInt}$. The second inequality follows from item a) of Proposition \ref{prop:alt2}.  

Proof of c): Using a), for $t \geq t_{idInt}$, $N_{idInt}= R_{idInt} \leq R_{idContMin} =N_{idContMin}$, and the first inequality follows from item b3) of Proposition \ref{prop:alt2}.\\

\noindent \emph{Proof of Proposition \ref{prop:ref}.} Proof of a1): to see that $R_{Int}(t) \leq R_{del-MTD}(t)$, note that as long as tumor size is lower than $N_{tol}$, both treatments coincide, then apply 
Proposition \ref{prop:alt0} from that point on. The other inequalities have already been proved. The proof of a2) is similar.  

Proofs of b) and c): in b), the inequality $t_{del-idMTD} \leq t_{idInt}$ follows from item b1) of Proposition \ref{prop:alt2}, applied from the (common) time when tumor size reaches $N_{tol}$ under both treatments, other inequalities were shown already. The proof of c) is as in Viossat and Noble.\\

\noindent \emph{Proof of Proposition \ref{prop:S}.}  We first need a lemma. 

\begin{lemma} \label{lem:S} 
a) Let $N^{\ast} \geq 0$. Consider a solution $(N, R)$ of \ref{VA_modelbis} under a treatment such that $L(t) = L_{max}$ whenever $N(t) > N^{\ast}$. 
Let $\bar{t} \geq 0$ be such that $N(\bar{t}) \geq N_{tol}$. If the sensitive population decreases when treated at 
$L_{max}$, then for all $t \geq \bar{t}$, $S(t) \leq S(\bar{t})$. If moreover $N(t) \geq N^{\ast}$ for all $t \geq \bar{t}$, 
then $S$ is non-increasing on $[\bar{t}, +\infty)$. 

b) Under containment (respectively, containment at $N_{min}$), once tumor size reaches $N_{tol}$ for the first time (respectively, $N_{min}$), the sensitive population is non-increasing. 
\end{lemma}
\begin{proof} a) The idea is that when $N > N^{\ast}$, $S$ is non-increasing by assumption, 
and when $N(t) \leq N^{\ast}$ for $t > \bar{t}$, the sensitive population must have decreased 
since time $\bar{t}$ because the resistant population increased (by assumption) and total tumor 
size did not. Formally, let $t \geq \bar{t}$. If for all $\tau$ in $(\bar{t}, t)$, 
$N(\tau) > N^{\ast}$, hence $L(\tau) = L_{max}$, then $S$ is non-increasing on $[\bar{t}, t]$ 
by assumption, therefore $S(t) \leq S(\bar{t})$. Otherwise, 
let $t_{max} = \max \{\tau \leq t, N(\tau) \leq N^{\ast} \}$. 
The previous argument implies that $S(t) \leq S(t_{max})$. Moreover, since $R$ is increasing,
\[S(t_{max}) = N(t_{max}) - R(t_{max}) \leq N_{tol} - R(t_{max}) \leq N(\bar{t}) - R(\bar{t}) = S(\bar{t})\]
Therefore, $S(t) \leq S(\bar{t})$. Finally, if for any $t_1 \geq \bar{t}$, $N(t_1) \geq N^{\ast}$, then the previous result applied from $t_1$ on shows that for any 
$t_2 \geq t_1$, $S(t_2) \leq S(t_1)$, hence $S$ is non-increasing on $[\bar{t}, \infty)$.

b) For containment, this follows from a) with $N^* = N_{tol}$ and the fact that once tumor size reaches $N_{tol}$ under containment, it never becomes smaller. The proof for containment at $N_{min}$ is the same with $N_{min}$ replacing $N_{tol}$. 
\end{proof}

We now prove Proposition \ref{prop:S}. Proof of a):  the first inequality is trivial since $S_{idMTD}=0$ (we only mentioned it to show that all inequalities from Eq. \ref{eq:bonus} are reversed). 
The inequality $S_{MTD}(t) \leq S_{alt}(t)$  follows from Lemma \ref{L1L2}, the fact that $\tilde{N}_{MTD}(r) \leq \tilde{N}_{alt}(r)$ 
(see Eq. \eqref{eq:tildeN}), 
and the fact that $\tilde{S}_{MTD}(r)$ is non-increasing by Assumption (A2). 
The last inequality follows from Assumption (A1), or, independently of (A1), from Lemma \ref{L1L2}, 
the fact that $\tilde{N}_{Cont}(r) \leq \tilde{N}_{noTreat}(r)$ (see Eq. \eqref{eq:tildeN}), and that once Containment starts treating, 
$S_{Cont}$ is non-increasing (Lemma \ref{lem:S}, item b)). 

Let us now prove that $S_{alt}(t) \leq S_{Cont}(t)$. 
Let $r_1 = \min\{r\geq R_0, \tilde{N}_{Cont}(r) = N_{tol}\}$. For $t \leq t_{Cont}(r_1)$, containment does not treat so 
$S_{Cont}(t) = S_{noTreat}(t) \geq S_{alt}(t)$ by Assumption (A1). 
Moreover, on $[r_1, R_{Cont}^{\infty})$, $\tilde{N}_{Cont}(r) \geq \tilde{N}_{alt}(r)$ and $\tilde{N}_{Cont}(r) \geq N_{tol}$, 
therefore $\tilde{S}_{Cont}(r)$ is non-increasing by Lemma \ref{lem:S}. Thus, by Lemma \ref{L1L2}, 
$S_{Cont}(t) \geq S_{alt}(t)$ for any $t$ in $[t_{Cont}(r_1), t_{alt}(R_{Cont}^{\infty}) \,)$. 

Finally, let $t \geq \max(t_{Cont}(r_1), t_{alt}(R_{Cont}^{\infty}))$, 
that is, such that $R_{Cont}(t) \geq r_1$ and $R_{alt}(t) \geq R_{Cont}^{\infty}$. 
Since $N_{Cont}(t) \geq N_{tol}$ for all $t \geq t(r_1)$, it follows from Lemma \ref{lem:S} that $S_{Cont}$ is non-increasing on $[t(r_1), +\infty[$, 
so $S_{Cont}(t) \geq S_{Cont}^{\infty}$. Thus, it suffices to show that $S_{alt}(t) \leq S_{Cont}^{\infty}$. There are two cases. 

Case 1: If $\tilde{N}_{alt}(R_{Cont}^{\infty}) \geq N_{tol}$, then by Lemma \ref{lem:S}, 
for all $t \geq t_{alt}(R_{Cont}^{\infty})$,  
\[S_{alt}(t) \leq S_{alt}(t_{alt}(R_{Cont}^{\infty}))= \tilde{S}_{alt} (R_{Cont}^{\infty}) \leq S_{Cont}^{\infty}\] 
where the last inequality follows from the fact that for $r < R_{Cont}^{\infty}$, $\tilde{S}_{alt}(r) \leq \tilde{S}_{Cont}(r)$ due to Eq. (1), so that 
\[\tilde{S}_{alt}(R_{Cont}^{\infty}) = \lim_{r \to R_{Cont}^{\infty}} \tilde{S}_{alt}(r) \leq \lim_{r \to R_{Cont}^{\infty}} \tilde{S}_{Cont}(r) = \lim_{t \to +\infty} S_{Cont}(t) = S_{Cont}^{\infty}\]

Case 2: If $\tilde{N}_{alt}(R_{Cont}^{\infty}) < N_{tol}$, then as long as $N_{alt}(t) \leq N_{tol}$, 
\[S_{alt}(t) \leq N_{tol} - R_{alt}(t) \leq N_{tol} - R_{Cont}^{\infty} \leq S_{Cont}^{\infty}\]
Moreover, if at some time $\bar{t}$, $N_{alt}(\bar{t}) = N_{tol}$ (which must indeed happen), then 
$S_{alt}(\bar{t}) \leq S_{Cont}^{\infty}$ by the previous argument, and for all $t \geq \bar{t}$, by Lemma \ref{lem:S}, 
$S_{alt}(t) \leq S_{alt}(\bar{t}) \leq S_{Cont}^{\infty}$. This concludes the proof of a).

Proof of b): The inequality $S_{ContNmin}(t) \leq S_{Int}(t)$ follows from Lemma \ref{L1L2}, the fact that $\tilde{N}_{contNmin}(r) \leq \tilde{N}_{Int}(r)$, and the fact that once tumor size reaches $N_{min}$, $S_{ContNmin}$ is non-increasing (Lemma \ref{lem:S}, item b)). The proof of $S_{Int}(t) \leq S_{Cont}(t)$ is as the proof of $S_{alt}(t) \leq S_{Cont}(t)$ (except that Assumption (A1) is not needed). Finally, the inequality $S_{del-MTD} \leq S_{Int}$ follows from $S_{MTD} \leq S_{alt}$ applied from the time at which tumor size reaches $N_{tol}$. 

Proof of c): 
we first prove $S_{idContNmin} \leq S _{idInt}$. Before tumor size reaches $N_{min}$, both treatments coincide, then until $t_{idContNmin}$,
 $N_{idContNmin} = N_{min} \leq N_{idInt}$ while $R_{idContNmin} \geq R_{idInt}$, so $S_{idContNmin} \leq S_{idInt}$. Finally, for $t \geq t_{idContNmin}$, $S_{idContNmin}(t)=0 \leq S_{idInt}(t)$. The proof of $S_{idInt} \leq S_{idCont}$ is similar. 

Proof of d): the first two inequalities are trivial, the third one was proved in c). The last inequality follows from (A1) but also, independently of (A1), from the following argument: for $t \leq t_{idCont}$, $N_{idCont} \leq N_{noTreat}$ while $R_{idCont} \geq R_{noTreat}$ by Proposition \ref{prop:ref}, so $S_{idCont} \leq S_{noTreat}$, and for $t \geq t_{idCont}$, $S_{idCont}=0$.

\section{Discussion}
\label{sec:disc} 

Viossat and Noble \cite{Viossat2021} provided qualitative conditions ensuring that a strategy aiming at containment, not elimination, 
minimizes resistance to treatment and is close to maximizing time to treatment failure. %
Some of these conditions were however debatable. In particular, their analysis did not allow for mutations from sensitive to resistant cells, 
a major concern of some key contributions to the field (Martin et al. 1992 \cite{Martin1992a}, Hansen et al. 2017 \cite{Hansen2017}), and did not apply to Norton-Simon models \cite{Norton1977}, 
which are standard to model chemotherapy. We showed how a refined analysis allows to handle these two issues. 
This suggests that containment strategies are likely to perform well in more general situations than was previously known. 

While Viossat and Noble compared across treatments the values of resistant and sensitive populations as a function of time, 
we first compare the induced trajectories in the R-N plane, that is, tumor sizes not at a given time, but when the resistant population reaches a given size. We made the additional assumption that the resistant population keeps increasing. This is consistent with the assumption that in the presence of fully resistant cells, 
the tumor is incurable, and is technically helpful (as the trajectory in the R-N plane is then the graph of a function), but we conjecture that our results hold without this assumption. 

What is crucial is that, all else being equal, a larger sensitive population leads to a lower resistant population growth-rate. For this reason, 
our analysis only allows for mutations from sensitive to resistant cells if an increase in the sensitive population size is more detrimental to 
the growth of the resistant population (through competition, or some other form of inhibition of resistant cells by sensitive cells) than it is beneficial 
(through mutations from sensitive to resistant cells). We show in Appendix \ref{app:mut} that this is typically the case for Gompertzian growth, or power-law models, at least in the absence of a strong resistance cost.\footnote{For logistic growth, our assumptions are likely to be valid if the variables $N$, $R$, $S$ are interpreted as densities, as in Strobl et al. 2020 \cite{Strobl2020a}, but not necessarily if they are interpreted as numbers of cells in the whole tumor, see Appendix \ref{app:mut}.}

There are however many other concerns with containment. Mutations, or phenotypic switching, could be modeled in other ways, 
and the fact that maintaining a relatively large tumor burden may lead to an accumulation of driver mutations remains a concern. 
Modeling patient death as occurring when the tumor reaches a critical size 
favors containment, and models in which the probability of death increases continuously 
with tumor size may lead to the conclusion that the expected survival time is lower under containment strategies than under more aggressive treatments (Mistry 2020\cite{Mistry2020}). Considering only two types of tumors cells is restrictive, 
and even with only two types, if resistant cells are only partially resistant, 
the logic changes, as the growth of resistant cells may be slowed down not only indirectly, through competition with sensitive cells, 
but also directly, through treatment effect. The impact of a containment strategy on the development of new metastases is also unclear. 
On the other hand, we did not consider additional benefits of containment, such as reduced treatment toxicity, less drug-induced mutations (Kuosmanen et al. 2021 \cite{Kuosmanen2021}) 
or a possible stabilization of tumor vasculature that could increase the efficiency of drug delivery (Enriquez-Navas et al. 2016 \cite{Enriquez-Navas2016}). 

This article should not be seen as providing unambiguous support for containment  strategies, but as part of a wider research program aiming at clarifying the conditions under which a strategy aiming at tumor stabilization is likely to perform better than a more aggressive treatment. Data allowing to fine-tune models is still scarce, but as new competition experiments are run, and new clinical trials open (NCT05393791, ACTOv/NCT05080556), more data should become available, allowing the community to reach more definite conclusions. 
\section{Acknowledgments}
This program has received funding from the European Union's Horizon 2020 research and innovation programme under the Marie Skłodowska-Curie grant agreement No 754362.
\appendix
\section{Mutations from sensitive to resistant cells}
\label{app:mut}
The analysis in this appendix is related to the work of Martin et al. (1992)\cite{Martin1992a}  and Hansen et al. (2017)\cite{Hansen2017}. Consider a basic Norton-Simon model with mutations (Norton and Simon 1977\cite{Norton1977}, Goldie and Coldman 1979\cite{Goldie1979}, Monro and Gaffney 2009\cite{Monro2009}):  
\begin{equation}
\begin{split}
\frac{dS}{dt} & =  g(N) \, (1 - L) S - \tau_1 g(N) S + \tau_2 g(N) R, \\
\frac{dR}{dt} & =  g(N) \, R + \tau_1 g(N) S - \tau_2  g(N) R,
\end{split}
\label{MG_mut}
\tag{Model 5}
\end{equation}
where $\tau_1$ and $\tau_2$ are mutation and backmutation rates. Taking $g(N)= \rho \ln(K/N)$ leads to a version of \ref{MG_model} with mutations: the original Monro and Gaffney model (Monro and Gaffney, 2009\cite{Monro2009}).

If the growth-rate function $g$ is decreasing in $N$, an increase in the size of the sensitive population leads to two opposite effects: it slows down the development of existing resistant cells (the competition effect), but usually increases the number of mutations from sensitive to resistant cells (the mutation effect). This trade-off has been studied by Martin et al. (1992\cite{Martin1992a}) and Hansen et al. (2017\cite{Hansen2017}). Here, we study whether such a model is compatible with our assumption that, during treatment, a larger sensitive population leads globally to a lower growth-rate of resistant cells. To do so, let $\phi_R$ denote the growth-rate function of resistant cells: 
$$\phi_R(S, R) =  g(N) R + \tau_1 g(N) S - \tau_2 g(N) R.$$
Denoting by $x_r = R/N$ the resistant fraction, it is easily checked that $\partial \phi_R / \partial S \leq 0$ if and only if: 
\begin{equation}
\label{eq:phis}
\frac{x_r}{\tau_1} (1-\tau_1 - \tau_2) + 1 \geq - \frac{g(N)}{Ng'(N)}.   
\end{equation}
Since the resistant fraction increases during treatment, this condition is bound to be hardest to satisfy at treatment initiation. 
The resistant fraction obtained from \ref{MG_mut} for the initial condition $S = 1$, $R=0$ is then (Goldie and Coldman, 1979\cite{Goldie1979}): 
\begin{equation} 
\label{eq:resfrac}
x_r = \frac{\tau_1}{\tau_1 +\tau_2} (1- N_0^{-\tau_1 - \tau_2}) \simeq \tau_1 \ln N_0,
\end{equation} 
where we used the approximation $N^{-\tau} \simeq 1 - \tau \ln N$ for $\tau$ small.  
Injecting \eqref{eq:resfrac} into \eqref{eq:phis} and using that $\tau_1$ and $\tau_2$ are much smaller than 1 leads to:
\begin{equation}
\label{eq:phis2}
\ln N_0  + 1 \geq - \frac{g(N_0)}{N_0g'(N_0)}.
\end{equation}  
Let us now consider various growth-models. 

Case 1 (power-law model): $g(N) = \rho N^{- \gamma}$ with $0 < \gamma < 1$. 
Eq. \eqref{eq:phis2} becomes: 
\[\ln N_0 + 1 \geq 1/\gamma.\] 
Typical choices for $\gamma$ are $\gamma = 1/3$ or $\gamma = 1/4$ (Gerlee, 2013\cite{Gerlee2013}; Benzekry et al., 2014\cite{Benzekry2014}; our $\gamma$ corresponds to  $1-\gamma$ in these references). The condition then holds by a huge margin for any detectable tumor size. 

Case 2 (Gompertzian growth): $g(N) = \rho \ln (K/N)$.
Eq \eqref{eq:phis2} becomes: 
\[\ln N_0 + 1  \geq \ln(K/N_0),\] 
which is satisfied if $K \leq e N^2$. Standard values of the carrying capacity in Gompertzian models are in the range $10^{12}-10^{13}$ (e.g., $K = 2 \times 10^{12}$ in Monro and Gaffney (2009)\cite{Monro2009}). Eq. \eqref{eq:phis2} is then satisfied for any detectable tumor size.

Case 3 (logistic growth) : $g(N) = \rho(1 - K/N)$. Eq \eqref{eq:phis2} becomes: 
\[\ln N_0 + 1 \geq \frac{K}{N_0}  - 1.\] 

This condition need not be satisfied, depending on the interpretation of the model and parameter choices.  For instance, Monro and Gaffney (2009)\cite{Monro2009} take $N_0 = 10^{10}$. Then $\ln N_0 \simeq 23$ and the condition is roughly $K \leq 2.5 \times 10^{11}$, which is not satisfied for standard values of the carrying capacity $K$.\footnote{The choice of carrying capacity may differ for a Gompertz or a logistic growth model. However, in Monro and Gaffney (2009)\cite{Monro2009}, the lethal tumor size is taken to be $5 \times 10^{11}$ so in a logistic growth version of their model, $K$ would have to be at least that large and Eq. \eqref{eq:phis2} would not be satisfied.} The condition would however be satisfied for larger initial tumor sizes, modeling late-stage treatments. Actually, when logistic growth models are used in the adaptive therapy literature, the initial tumor size is often assumed to be a large fraction of the carrying capacity (e.g., Zhang et al. 2017\cite{Zhang2017}, Strobl et al. 2020\cite{Strobl2020a}, Mistry 2020\cite{Mistry2020}). This may be interpreted as modeling late-stage treatments, or as a model of local growth. In the latter case, the carrying capacity should be seen as the maximal number of cells \emph{for the current tumor volume} (or equivalently, the variables $N$, $S$, $R$, $K$ should be interpreted as densities). Assuming $10^{10}$ tumor cells at tumor initiation, the estimate $x_r/\tau_1 \simeq 23$ would still be valid, and Eq. \eqref{eq:phis} would become $K/N \leq 25$, which is bound to be satisfied in a model of local growth. 

Let us now consider three variants of \ref{MG_mut}. 

\emph{Variant 1: birth-death model.} In \ref{MG_mut}, the number of mutations is assumed proportional to the net growth-rate of the tumor. It would be natural to assume that the number of mutations is proportional to the net birth-rate. This would  increase the effective mutation rate (that is, the average number of mutations relative to a given increase of tumor size).\footnote{This is the exact effect when the turnover rate $b(N) / (b(N) - d(N))$ is constant, where $b(N)$ and $d(N)$ are the birth and death rates.} 
However, since the condition we found is insensitive to the precise mutation rates $\tau_1$ and $\tau_2$, this is unlikely to affect the previous analysis. 

\emph{Variant 2: late-stage treatment.} The previous analysis is better suited for a first line treatment than a second or third line treatment, especially if resistance to the first treatment may be associated with resistance to ulterior ones. However, in such a situation (late-stage treatment), the initial resistant population is likely to be larger than the one given by \eqref{eq:resfrac}, and so condition \eqref{eq:phis} is more likely to be satisfied.  

\emph{Variant 3: cost of resistance in the baseline growth-rate.} Consider the following variant of \ref{MG_mut}, with a different growth-rate parameter for sensitive cells and for resistant cells: 
\begin{equation}
\begin{split}
\frac{dS}{dt} & = \rho_s g(N) \, (1 - L) S  - \tau_1 \rho_s g(N) S + \tau_2 \rho_r g(N) R, \\
\frac{dR}{dt} & = \rho_r g(N) \, R + \tau_1 \rho_r g(N) S - \tau_2 \rho_r g(N) R.
\label{MG_mut_cost}
\end{split}
\tag{Model 6}
\end{equation}
(the terms $g(N)$ in \ref{MG_mut} correspond here to terms of the form $\rho g(N)$, with $\rho = \rho_s$ or $\rho =\rho_r$.) The absolute growth-rate of resistant cells is now 
\[\phi_R(S, R) = \rho_r g(N) R  + \tau_1 \rho_s g(N) S - \tau_2 \rho_r g(N) R.\] 
The condition for $\partial \phi_R / \partial S$ to be nonpositive becomes: 
\begin{equation}
\label{eq:phircost}
\frac{x_r}{\tau_1} (\rho_r/\rho_s - \tau_2 \rho_r/\rho_s - \tau_1) + 1 \geq - \frac{g(N)}{N g'(N)}.
\end{equation}
Moreover, if $\rho_r$ is substantially smaller than $\rho_s$, then the resistant fraction at treatment initiation is no longer given by \eqref{eq:resfrac} but approximately by (see Viossat and Noble, 2020\cite{Viossat2020}, Section 7):
\begin{equation}
\label{eq:resfraccost}
x_r = \frac{\tau_1}{1- \rho_r/\rho_s}.
\end{equation}
Injecting \eqref{eq:resfraccost} into \eqref{eq:phircost} and using that $\tau_1$ and $\tau_2$ are much smaller than $1$ leads to the condition: 
\begin{equation}
\label{eq:condcost2}
\frac{1}{1 - \rho_r/\rho_s} \geq - \frac{g(N)}{N g'(N)}.
\end{equation} 
For a Power-law model, 
the condition becomes $\rho_r/\rho_s \geq 1 - \gamma$.  
For $\gamma = 1/3$, this is satisfied if and only if $\rho_r/\rho_s \geq 2/3$.  that is, if and only if the resistance cost is not too large. 

For a Gompertzian model, the right-hand side is $\ln(K/N)$ and Eq. \eqref{eq:condcost2} may be written: $\rho_r/\rho_s \geq 1- 1/\ln(K/N)$. With Monro and Gaffney's (2009) \cite{Monro2009}  values: $N_0 = 10^{10}$, $K=2 \times 10^{12}$, this is satisfied if $\rho_r/\rho_s \geq 0.81$. 

With logistic growth, the right-hand side is $K/N - 1$ and the condition may be written as $\rho_r/\rho_s \geq (K-2N)/(K-N)$.\footnote{This condition is approximately correct only if $\rho_r/\rho_s$ is substantially different from 1, which explains that taking the limit $\rho_r/\rho_s \to 1$ does not lead to the condition obtained in the absence of a resistance cost.}  Assuming for instance $\rho_r/\rho_s = 4/5$, this boils down to $K \leq 6 N$. This would not be satisfied at treatment initiation if $N$ and $K$ represent total numbers of cells in the whole tumor (except possibly for a late-stage treatment), but seems likely to be satisfied in a model of local growth.

We conclude that in the absence of resistance costs, our analysis applies to several standard models of tumor growth with mutations, such as Power-law models or Gompertzian growth, and possibly to logistic growth, at least when it models local growth. However, if the baseline growth-rate of resistant cells is substantially smaller than the baseline growth-rate of sensitive cells, our assumptions become more restrictive and might fail even for Gompertzian growth. 
This is in line with Hansen et al.'s (2017)\cite{Hansen2017} finding that, contrary to common wisdom, 
a resistance cost in the baseline growth rate may make it 
less likely that containment strategies outperform 
more aggressive treatments. Note however that the fact that 
we can no longer prove that containment outperforms MTD does 
not mean that it would not do so. Moreover, the analysis 
of Viossat and Noble (2020)\cite{Viossat2020}, Section 7 
of the supplementary material, suggests that the mutation 
effect could only make MTD marginally superior to containment. 
  
\section{Comparison principles}
\label{app:cp}
The following comparison principle is standard: 
\begin{prop}
\label{prop:cpstandard}
Let $\Omega$ be a nonempty open subset of $\R^2$. Let $f : \Omega \to \R$ be continuously differentiable. Consider the ordinary differential equation $x'(t) = f(t, x(t))$. Let  $t_1 \geq t_0$. Let $u : [t_0, t_1] \to \R$ be solution of this ODE and let $v : [t_0, t_1] \to \R$ be a subsolution. That is, $v$ is continuous, almost everywhere differentiable, $(t, v(t)) \in \Omega$ on $[t_0, t_1]$, and, almost everywhere, $v'(t) \leq f(t, v(t))$. If furthermore $v(t_0) \leq u(t_0)$, then $v(t) \leq u(t)$ for all $t \in [t_0, t_1]$. 
\end{prop}
We want to apply a variant of this result to two equations: first, 
\[
\frac{d\tilde{N}}{dr} = G(r, \tilde{N}(r)) \mbox{ where } G(r, \tilde{N}) = \frac{f_N(\tilde{N}, r, \tilde{L}_2(r))}{f_R(\tilde{N}, r)},
\]
(here $\tilde{N}(r)$ plays the role of $u(t)$, and $G$ the role of $f$ in Proposition \ref{prop:cpstandard}). Second, 
\[\frac{dR}{dt} = H(R(t)) \mbox{ where } H(R) = f_R(\tilde{N}_2(R), R), \]
where $\tilde{L}_2$ and $\tilde{N}_2$ are not continuously differentiable (otherwise Proposition \ref{prop:cpstandard} would directly apply), but piecewise $C^1$ in the strong sense we defined in Section \ref{sec:model}. For instance, for $\tilde{L}_2$, which is defined on $[R_0, R_2^{\infty})$, there exist values $r_0 = R_0 < r_1 < ... < r_n = R_2^{\infty}$ such that for each $i$ in $\{1,.., n-1\}$, $\tilde{L}_2$ coincides on $[r_i, r_{i+1})$ with a continuously differentiable function $\tilde{L}_2^i$ defined on a neighborhood of $[r_i, r_{i+1})$. 

We thus need variants of Proposition \ref{prop:cpstandard} where $f$ is slightly less regular. The proof of these variants consists in repeated applications of Proposition \ref{prop:cpstandard}. 

\begin{prop}
\label{prop:cp}  
The conclusion $u(t) \leq v(t)$ on $[t_0, t_1]$ of Proposition \ref{prop:cpstandard} still holds in the following cases: 

a) if $f$ takes the form $f(t, x) = \psi(t, x, \phi(t))$, where $\psi$ is continuously differentiable and $\phi$ is strongly piecewise $C^1$.  


b) if $f$ is a function of $x$ only ($f : I \to \R$, where $I$ is a nonempty interval) which 
is strongly piecewise $C^1$, 
and $u$ and $v$ are strictly increasing. 
\end{prop}
\begin{proof}
Proof of a): by assumption, there exists an integer $n$ and real numbers $(\tau_k)_{0 \leq k \leq n}$ satisfying $t_0 = \tau_0 < \tau_1 < ... < \tau_n = t_1$ such that on $A_k = [t_k, t_{k+1}) \times \R$, $f$ coincides with a continuously differentiable function $f_k$ defined on a neighborhood of $A_k$. Assuming $u(\tau_k) \leq v(\tau_k)$, Proposition 
\ref{prop:cpstandard} implies that $u(t) \leq v(t)$ on $[t_k, t_{k+1})$ and this is also true at time $t_{k+1}$ by continuity of $u$ and $v$. An induction argument then gives the result. 

Proof of b): By assumption, there exist a finite number of values $x_0$,...,$x_n$ such that $f$ coincides on $[x_k, x_{k+1})$ with a continuously differentiable function $\phi_k$ defined on a neighborhood of $[x_k, x_{k+1})$, and we may assume (up to adding an artificial point) that $x_0 \leq u(t_0)$. Since $u$ and $v$ are strictly increasing, there exists an integer $q \leq 2(n+1)$ and a sequence of times $\tau_0 = t_0 < \tau_1 < ... < \tau_q = t_1$  such that on $(\tau_j, \tau_{j+1})$, $u(t)$ and $v(t)$ are never equal to one of the $x_k$. Assuming $u(\tau_j) \leq v(\tau_j)$, then either: for all $t$ in  $[\tau_j, \tau_{j+1})$, $u(t)$ and $v(t)$ belong to the same interval $[x_k, x_{k+1})$ and Proposition \ref{prop:cpstandard} implies that $u(t) \leq v(t)$ on $[\tau_j, \tau_{j+1})$; or there exists $k$ such that for all $t$  in $[\tau_j, \tau_{j+1})$, $u(t) \leq x_k \leq v(t)$ and the same inequality holds trivially. In both cases,  $u(\tau_{k+1}) \leq v(\tau_{k+1})$ by continuity of $u$ and $v$. An induction argument then gives the result. 
\end{proof}

\bibliography{biblioVA}
\bibliographystyle{unsrt}
\end{document}